\documentclass[12pt]{amsart}
\usepackage{amsmath,amssymb,amsthm,amscd}
\usepackage[margin=1.15in]{geometry}
\usepackage[colorlinks=true,linkcolor=blue,citecolor=blue,urlcolor=blue]{hyperref}
\newtheorem{theorem}{Theorem}[section]
\newtheorem{lemma}[theorem]{Lemma}

\newtheorem{corollary}[theorem]{Corollary}
\theoremstyle{remark}
\newtheorem{remark}[theorem]{Remark}
\DeclareMathOperator{\CH}{CH}
\DeclareMathOperator{\Ch}{Ch}
\DeclareMathOperator{\Spin}{Spin}
\DeclareMathOperator{\im}{Im}
\DeclareMathOperator{\cl}{cl}
\newcommand{\F}{\mathbb F_2}
\newcommand{\St}{\operatorname{St}}
\newcommand{\Sq}{\operatorname{Sq}}

\title[Integral Weyl invariants of spin and Clifford groups]{Integral Weyl Invariants in Chow Characteristic Images of Spin and Special Clifford Groups}
\author{Sanghoon Baek}

\address{Department of Mathematical Sciences, KAIST, 291 Daehak-ro,
Yuseong-gu, Daejeon 34141, Republic of Korea}
\email{sanghoonbaek@kaist.ac.kr}
\urladdr{https://mathsci.kaist.ac.kr/~sbaek/}

\thanks{This work was supported by the National Research Foundation of Korea (NRF) grant funded by the Korean government (MSIT) (No.~RS-2026-25471364).}
\date{}
\hypersetup{
  pdftitle={Integral Weyl Invariants in Chow Characteristic Images of Spin and Special Clifford Groups},
  pdfauthor={Sanghoon Baek},
  pdfkeywords={Chow ring of a classifying space, spin group, special Clifford group, Weyl invariants, Steenrod operations, integral Hodge conjecture, algebraic cycles modulo torsion}
}

\begin{document}

\begin{abstract}
Let $G=\Spin(n)$ be the split spin group over an arbitrary field, with
$n\ge7$. Extending a Steenrod-theoretic obstruction of Karpenko, we
classify the recursively defined integral Weyl invariants $q_i$ in the
Benson--Wood generating set that lie in the Chow characteristic image:
the only such invariant is $q_3$ for $\Spin(10)$. We obtain the analogous
classification for the recursive invariants $f_i$ of the special Clifford
group $\Gamma^+(n)$: in their finite generating range, the only such
invariant is $f_2$ for $\Gamma^+(7)$. Over $\mathbb C$, the class
corresponding to $q_i$ in the torsion-free quotient of the integral
cohomology of the classifying space $BG$ is algebraic precisely when
$(n,i)=(10,3)$. For each $n$, a single smooth projective approximation
simultaneously realizes all the corresponding classes in the finite
range. Every nonexceptional class remains nonalgebraic after the addition
of any torsion class, as detected by a Bockstein--Steenrod operation.
\end{abstract}

\maketitle

\section{Introduction}
Let $G=\Spin(n)$ be the split spin group over a field $k$, with
$n\ge7$, and let $T\subset G$ be a split maximal torus with Weyl group $W$. We use
Totaro's Chow ring $\CH(BG)$ of the classifying space $BG$
\cite{TotaroChow}. Restriction along $T\subset G$ gives the Chow
characteristic homomorphism
\[
  \Phi_n:\CH(BG)\longrightarrow \CH(BT)^W.
\]
Let $T^*$ denote the character group of $T$, and let
$\operatorname{Sym}(T^*)$ denote the symmetric $\mathbb Z$-algebra on
$T^*$. When discussing
singular cohomology, we use the same symbols $G$ and $T$ for the
corresponding compact Lie group and its maximal torus. For every
$d\ge0$, algebraic and topological first Chern classes give canonical
$W$-equivariant identifications
\[
  \CH^d(BT)
  \simeq \operatorname{Sym}^d(T^*)
  \simeq H^{2d}(BT;\mathbb Z).
\]
Thus $\Phi_n$ takes values in an explicitly computable ring of Weyl
invariants.

Choose the standard characters $x_1,\ldots,x_m,z\in T^*$, where
$m=\lfloor n/2\rfloor$ and $2z=x_1+\cdots+x_m$. Then
\[
  \CH(BT)=\mathbb Z[z,x_1,\ldots,x_m]/(2z-x_1-\cdots-x_m).
\]
Set $c_j=e_j(x_1,\ldots,x_m)$,
$p_j=e_j(x_1^2,\ldots,x_m^2)$, and
$P_m:=\mathbb Z[p_1,\ldots,p_m]$. Set $r(n)=m-3$ when
$4\mid n$, and $r(n)=m-2$ otherwise. Benson and Wood proved that $\CH(BT)^W$ is generated over $P_m$ by the
recursively defined invariants $q_1,\ldots,q_{r(n)}$, together with the
Euler class $c_m$ when $n$ is even and one additional generator. By the
orbit product of $z$, we mean the product of the distinct elements in the
$W$-orbit of $z$. The additional generator is a square root of this orbit
product when $n$ is odd or divisible by $4$, and is the orbit product
itself when $n\equiv2\pmod4$
\cite[Theorem~7.1]{BensonWood}.

Several of the other generators are already understood. The Pontryagin
classes belong to $\im\Phi_n$ since, up to signs, they are the images of
the even Chern classes of the standard orthogonal representation. The
orbit product of $z$ is the image under $\Phi_n$ of the top Chern class
of the spin representation when $n$ is odd and of a half-spin
representation when $n$ is even. It therefore belongs to $\im\Phi_n$.
Karpenko proved that the additional square-root generator does not lie
in the image for odd $n\ge7$ or for $n\equiv0\pmod4$
\cite[Theorems~2.2 and 3.2]{KarpenkoEnvelopes}. Consequently, the
additional Benson--Wood generator lies in $\im\Phi_n$ precisely when
$n\equiv2\pmod4$. For the Euler class, Karpenko proved
$c_m\notin\im\Phi_n$ for even $n\ne10$, whereas
$c_5\in\im\Phi_{10}$; see
\cite[Theorem~3]{KarpenkoClassifying}, as corrected in
\cite[Appendix~A, especially Lemma~A.1]{KarpenkoModuloTorsion}.

The algebraic characteristic map $\Phi_n$ contrasts sharply with the
topological restriction map
\begin{equation}\label{eq:topological-restriction}
  H^*(BG;\mathbb Z)/\mathrm{torsion}
  \longrightarrow H^*(BT;\mathbb Z)^W.
\end{equation}
The rational restriction theorem shows that
\eqref{eq:topological-restriction} is injective. Benson and
Wood proved that it is surjective exactly when
$n\not\equiv3,4,5\pmod8$; in the remaining congruence classes, the additional generator is replaced by its double and its square. In
particular, every $q_i$ in the range $1\le i\le r(n)$ lies in the image
\cite[Theorem~10.2]{BensonWood}.

Using Steenrod operations, Karpenko proved that
\[
  2^i+1<\frac n2
  \quad\Longrightarrow\quad
  q_i\notin\im\Phi_n
\]
\cite[Theorem~3]{KarpenkoClassifying}. We complete the classification beyond
this numerical range and determine all the $q_i$ appearing in Benson and
Wood's generating set.

\begin{theorem}\label{thm:classification}
Let $G=\Spin(n)$ with $n\ge7$, and let $1\le i\le r(n)$. Then
\[
  q_i\in\im\Phi_n
  \quad\Longleftrightarrow\quad
  (n,i)=(10,3).
\]
\end{theorem}

The exceptional inclusion uses the Euler class and the torsion index; all
other cases are detected by explicit Steenrod operations. Together with the
known results above, the theorem determines exactly which generators in
Benson and Wood's standard generating set lie in $\im\Phi_n$.

Now assume that $k=\mathbb C$. For every $d\ge0$, let $\cl:\CH^d(BG)\longrightarrow H^{2d}(BG;\mathbb Z)$ denote the integral cycle-class map. Each $q_i$, $1\le i\le r(n)$, has a
unique preimage in
$H^{2^{i+1}}(BG;\mathbb Z)/\mathrm{torsion}$ under
\eqref{eq:topological-restriction}. We call a class $u\in H^{2d}(BG;\mathbb Z)/\mathrm{torsion}$ \emph{algebraic modulo torsion} if it lies in the image of the composite
\[
  \CH^d(BG)\xrightarrow{\cl}H^{2d}(BG;\mathbb Z)
  \longrightarrow H^{2d}(BG;\mathbb Z)/\mathrm{torsion}.
\]
We use the same terminology for classes on smooth complex varieties.

\begin{corollary}\label{cor:topological-lifts}
Assume that the base field is $\mathbb C$. For $1\le i\le r(n)$, let
$u_i\in H^{2^{i+1}}(BG;\mathbb Z)/\mathrm{torsion}$ be the unique class
restricting to $q_i$. Then $u_i$ is algebraic modulo torsion if and
only if $(n,i)=(10,3)$.
\end{corollary}

\begin{proof}
Naturality of the cycle-class maps gives a commutative square relating
$\Phi_n$ to topological restriction. Any torsion class restricts to zero in
$H^*(BT;\mathbb Z)$, which is torsion-free. Hence an algebraic
representative of $u_i$ modulo torsion restricts exactly to $q_i$.
Conversely, the injectivity of \eqref{eq:topological-restriction}
identifies the cycle class of any Chow lift of $q_i$ with $u_i$ modulo
torsion. Thus $u_i$ is algebraic modulo torsion if and only if
$q_i\in\im\Phi_n$, and the result follows from
Theorem~\ref{thm:classification}.
\end{proof}

The Steenrod operations used to prove Theorem~\ref{thm:classification}
also yield a smooth projective realization of this algebraicity
classification.

\begin{corollary}\label{cor:hodge-counterexamples}
For every $n\ge7$, there exist a smooth projective complex variety $X_n$
and non-torsion integral Hodge classes
\[
  \alpha_{n,i}\in H^{2^{i+1}}(X_n;\mathbb Z),
  \qquad 1\le i\le r(n),
\]
such that $\alpha_{n,i}$ is algebraic modulo torsion if and only if
$(n,i)=(10,3)$. In the exceptional case, $\alpha_{10,3}$ itself may be
chosen algebraic.
\end{corollary}

Cohomology-operation obstructions to algebraicity originate with Atiyah--Hirzebruch and were developed further by Totaro
\cite{AtiyahHirzebruch,TotaroCobordism}. Non-torsion cohomology classes
that remain non-algebraic modulo torsion in codimension~$2$ were later
constructed in related Hodge, Tate, and classifying-space settings by
Pirutka--Yagita, Antieau, and Kameko
\cite{PirutkaYagita,AntieauTate,KamekoCycleMap}. Here all the $q_i$ for a fixed $n$ are treated uniformly and realized
simultaneously on a single smooth projective variety.

Let $\widetilde G=\Gamma^+(n)$ be the split special Clifford group over
$k$, and let $\widetilde T\subset\widetilde G$ be a split maximal torus
containing $T$. Its Weyl group is naturally identified with $W$. Write
\[
  \widetilde\Phi_n:\CH(B\widetilde G)\longrightarrow
  \CH(B\widetilde T)^W
\]
for the characteristic homomorphism. Characteristic images for even special
Clifford groups were studied by Karpenko. In particular, the image for
$\Gamma^+(12)$ was determined, and general constraints in higher even ranks
were obtained \cite[Theorems~3.16 and 4.9]{KarpenkoSpecialClifford}.
Following Karpenko and Merkurjev
\cite[\S4, formulas~(4.1)--(4.2)]{KarpenkoMerkurjevSpin}, we consider
recursively defined $W$-invariants $f_i\in\CH(B\widetilde T)^W$. We recall the recursion and fix the resulting representatives in
Section~5. The members of this family occurring in the corresponding finite
generating set are $f_1,\ldots,f_{s(n)}$, where $s(n)=m-1$ for odd $n$
and $s(n)=m-2$ for even $n$; see
\cite[Proposition~4.3, Proposition~5.1, and
Corollary~5.2]{KarpenkoMerkurjevSpin}.

\begin{theorem}\label{thm:clifford-classification}
For $n\ge7$ and $1\le i\le s(n)$, one has
\[
  f_i\in\im\widetilde\Phi_n
  \quad\Longleftrightarrow\quad
  (n,i)=(7,2).
\]
\end{theorem}

The non-inclusions follow by restriction to the spin torus and the common
coefficient recurrence developed in Section~3; the exceptional inclusion for
$\Gamma^+(7)$ uses the fourth Chern class of the spin representation and the
torsion index. Sections~4 and~5 then prove the spin and special Clifford
classifications, respectively, with the Hodge-theoretic consequences included
in Section~4.

\section{Preliminaries}
We use the notation of the Introduction. All Chow rings are graded by
codimension; thus $\deg c_j=j$ and $\deg p_j=2j$, and the formal
polynomial rings below carry the same grading. We put $c_0=p_0=1$.
In rank $m$, one has $c_j=p_j=0$ for $j>m$. In formal polynomial
calculations, however, these truncation relations are imposed only when
the result is interpreted in rank $m$.

Write $\Ch(-):=\CH(-)/2\CH(-)$. Reduction modulo two gives
$c_1=x_1+\cdots+x_m=0$ in $\Ch(BT)$, while
$c_2,\ldots,c_m$ remain algebraically independent
\cite[Example~3.1]{KarpenkoModuloTorsion}.

Let $\St=\sum_{r\ge0}\St^r$ denote the total cohomological Steenrod
operation on mod-two Chow groups. Its $r$th component is a natural
operation
\[
  \St^r:\Ch^d(-)\longrightarrow\Ch^{d+r}(-).
\]
In characteristic different from $2$, we use Brosnan's operations, and
in characteristic $2$, Primozic's motivic construction
\cite{Brosnan,Primozic}. On classifying spaces, these operations are
defined degreewise using Totaro's finite-dimensional approximations
\cite{TotaroChow}.

Let $\rho:\CH(BT)^W\hookrightarrow\CH(BT)\longrightarrow\Ch(BT)$ be the composite of the inclusion with reduction modulo two. Naturality
of Steenrod operations gives a commutative square whose horizontal
arrows are induced by the inclusion $T\hookrightarrow G$:
\[
\begin{CD}
  \Ch(BG) @>>> \Ch(BT)\\
  @V{\St}VV     @VV{\St}V\\
  \Ch(BG) @>>> \Ch(BT).
\end{CD}
\]
Indeed, suppose that $x=\Phi_n(a)$ and let $\bar a$ be the reduction
of $a$. Choose $b\in\CH(BG)$ whose reduction modulo two is
$\St(\bar a)$. Naturality gives $\rho\bigl(\Phi_n(b)\bigr)=\St\bigl(\rho(x)\bigr)$. Since $\Phi_n(b)\in\CH(BT)^W$, it follows that
\begin{equation}\label{eq:steenrod-stability-criterion}
  x\in\im\Phi_n
  \quad\Longrightarrow\quad
  \St(\rho(x))\in\im\rho.
\end{equation}
This is the Steenrod-stability criterion used in Karpenko's argument
\cite[Proposition~1]{KarpenkoClassifying}. We shall mainly use
$\St^1$; in the small-rank cases below, we also use $\St^3$ and
$\St^7$.

For computations, recall that $\St$ is a ring homomorphism and sends
the first Chern class $u$ of every torus character to $u(1+u)$. Its
homogeneous components satisfy the Cartan formula
\begin{equation}\label{eq:cartan}
  \St^r(ab)=\sum_{p+q=r}\St^p(a)\St^q(b).
\end{equation}
For $0\le r\le j$, the Wu--Borel formula in $\Ch(BT)$ is
\begin{equation}\label{eq:steenrod-cj}
  \St^r(c_j)
  =\sum_{k=0}^{r}\binom{r-j}{k}c_{r-k}c_{j+k},
\end{equation}
where the generalized binomial coefficient is taken modulo $2$
(with $\binom{-d}{k}=(-1)^k\binom{d+k-1}{k}$ for $d>0$). For
$r>j$, one has $\St^r(c_j)=0$. See \cite{WuSquares}, \cite[Th\'eor\`eme~7.1]{BorelSteenrod}, and \cite[\S4]{KarpenkoModuloTorsion}.

Taking $r=1$ in \eqref{eq:steenrod-cj} and using $c_1=0$ gives
\begin{equation}\label{eq:St1-cj}
  \St^1(c_j)=c_{j+1}\quad(j\text{ even}),
  \qquad
  \St^1(c_j)=0\quad(j\text{ odd}).
\end{equation}

For the remainder of this section, assume that the base field is
$\mathbb C$. We write $\cl$ for either the integral or the mod-two
cycle-class map, as determined by its source and target, and use a bar
for reduction modulo two. The mod-two cycle-class map gives a canonical
$W$-equivariant identification $\Ch^d(BT)\simeq H^{2d}(BT;\mathbb F_2)$, compatible with the integral identification in the Introduction. Under
these identifications, $\rho$ corresponds to reduction modulo two.

Let $\beta:H^k(-;\mathbb F_2)\longrightarrow H^{k+1}(-;\mathbb Z)$ denote the Bockstein homomorphism associated with
\[
  0\longrightarrow\mathbb Z\xrightarrow{\times2}\mathbb Z
  \longrightarrow\mathbb F_2\longrightarrow0.
\]
For every smooth complex variety $Y$, compatibility of algebraic and
topological Steenrod operations gives
\begin{equation}\label{eq:brosnan-comparison}
  \cl\bigl(\St^j(a)\bigr)
  =\Sq^{2j}\bigl(\cl(a)\bigr),
  \qquad a\in\Ch^d(Y);
\end{equation}
see \cite[Corollary~9.12]{Brosnan}. The same identity applies to
classifying spaces degreewise via smooth finite-dimensional
approximations.

Consequently, if $a\in\CH^d(Y)$ and
$\alpha=\cl(a)\in H^{2d}(Y;\mathbb Z)$, then
\begin{equation}\label{eq:algebraic-bockstein-vanishing}
  \beta\Sq^{2j}(\bar\alpha)=0
  \qquad(j\ge0).
\end{equation}
Indeed, choose $w\in\CH^{d+j}(Y)$ lifting $\St^j(\bar a)$. Then
\eqref{eq:brosnan-comparison} gives $\Sq^{2j}(\bar\alpha)=\overline{\cl(w)}$, and exactness of the Bockstein sequence proves
\eqref{eq:algebraic-bockstein-vanishing}.

\section{Image and coefficient constraints}
We record two general tools used in both the spin and special Clifford cases.

\begin{lemma}\label{lem:odd-image-constraint}
Set $R_m:=\F[c_2,\ldots,c_m]\subset\Ch(BT)$, and let
$u\in R_m \,\cap\,\im\rho$ be homogeneous of odd degree. If $n$ is odd or
$4\mid n$, then $u=0$; if $n=2m$ with $m$ odd, then $u\in c_mR_m$.
\end{lemma}

\begin{proof}
By the Benson--Wood presentation recalled in the Introduction,
$\im\rho$ is generated by the reductions of the Pontryagin classes,
the $q_i$, and the orbit generator, all of which have even degree,
together with the Euler generator $e=c_m$ of degree $m$ when $n$ is
even.

If $n$ is odd, the Euler generator does not occur. If $4\mid n$, then
$m$ is even. Hence, in either case, $\im\rho$ is generated by
homogeneous elements of even degree and therefore has no nonzero
homogeneous element of odd degree. Thus $u=0$.

Suppose that $n=2m$ with $m$ odd. Then $c_m$ is the only odd-degree
generator of $\im\rho$, while all the other generators have even
degree. Choose a homogeneous polynomial expression for $u$ in these
generators. Every monomial occurring in this expression must contain
a factor $c_m$. Hence $u\in c_m\Ch(BT)^W$. By \cite[Proposition~3.2]{KarpenkoModuloTorsion}, $\Ch(BT)^W=R_m[\theta]$, where $\theta$ is the orbit product of $z$ in this case. As a
polynomial in $z$, it is monic of positive degree.

The powers of $\theta$ are linearly independent over $R_m$. Indeed,
since $\Ch(BT)\simeq\F[z,x_2,\ldots,x_m]$, the element $z$ is algebraically independent over $R_m$. If
$a_d\ne0$, then the leading term in $z$ of $\sum_{j=0}^d a_j\theta^j$ comes from $a_d\theta^d$ and has nonzero coefficient $a_d$. The sum is
therefore nonzero. Thus every element of $R_m[\theta]$ has a unique
expansion in powers of $\theta$.

Since $u$ belongs to both $R_m$ and $c_mR_m[\theta]$, write
\[
  u=c_m\sum_{j=0}^d a_j\theta^j,
  \qquad a_j\in R_m.
\]
Because $u\in R_m$, its expansion in powers of $\theta$ is constant.
Uniqueness therefore gives $c_ma_j=0$ for every $j>0$. Since $R_m$ is
a polynomial ring, it is a domain, and $c_m\ne0$. Hence $a_j=0$ for
$j>0$, and consequently $u=c_ma_0\in c_mR_m$, as required.
\end{proof}

Write $C(t)=\prod_{k=1}^m(1+x_kt)=\sum_{j\ge0}c_jt^j$. Since $C(t)C(-t)
  =\prod_{k=1}^m(1-x_k^2t^2)
  =\sum_{j\ge0}(-1)^jp_jt^{2j}$, comparison of the coefficients of $t^{2j}$ gives
\begin{equation}\label{eq:pj-cj}
  p_j=c_j^2+2\delta_j,
  \qquad
  \delta_j=\sum_{s=1}^j(-1)^s c_{j-s}c_{j+s}.
\end{equation}
This is the standard relation between the elementary symmetric
functions in the $x_k$ and those in the $x_k^2$; see
\cite[\S2]{BensonWood}. In formal polynomial calculations, we use the
right-hand side of \eqref{eq:pj-cj} to regard $p_j$ as a polynomial in
$c_1,c_2,\ldots$.

For a polynomial $H$ and a monomial $M$, write $[M]H$ for the
coefficient of $M$. For
$H\in\mathbb Z[c_1,c_2,\ldots]$, set
\[
  \epsilon(H):=\overline H\big|_{c_1=0}
  \in\F[c_2,c_3,\ldots],
\]
where the bar denotes coefficientwise reduction modulo two, and let
$H(p)$ denote the polynomial obtained by substituting $p_j$ for $c_j$.
Since $p_j\equiv c_j^2\pmod2$, the Frobenius identity in characteristic $2$ gives $H(p)\equiv H^2\pmod2$. Thus $H(p)-H^2$ is coefficientwise divisible by $2$, and we define
\[
  \mathcal D(H)
  :=\epsilon\!\left(\frac{H(p)-H^2}{2}\right).
\]

\begin{lemma}\label{lem:common-coefficient}
For $i\ge3$, let $H_i\in\mathbb Z[c_1,c_2,\ldots]$ be homogeneous
of degree $2^i$. Suppose that
$\epsilon(H_{i+1})=\mathcal D(H_i)$ for every $i\ge3$ and that
$[c_2c_6]\epsilon(H_3)=1$. Then, for every $i\ge4$,
\[
  {}[c_3c_6^{(2^i-2)/6}]\St^1\bigl(\epsilon(H_i)\bigr)=1
  \quad(i\text{ odd}),
  \qquad
  {}[c_5c_6^{(2^i-4)/6}]\St^1\bigl(\epsilon(H_i)\bigr)=1
  \quad(i\text{ even}).
\]
\end{lemma}

\begin{proof}
Fix $i\ge3$ and write
\[
  H_i=\sum_M\alpha_M M,
  \qquad
  M=\prod_{r=1}^{2^i}c_r^{a_r},
  \qquad
  a_r\in\mathbb Z_{\ge0},
  \qquad
  \sum_{r=1}^{2^i}ra_r=2^i.
\]
Fix any total order on the monomials occurring in $H_i$. Since $H_i^2=\sum_M\alpha_M^2M^2+
2\sum_{M<N}\alpha_M\alpha_NMN$, we have
\begin{equation}\label{eq:Hi-decomposition}
  \frac{H_i(p)-H_i^2}{2}
  =\sum_M\alpha_M\frac{M(p)-M^2}{2}
  +\sum_M\frac{\alpha_M-\alpha_M^2}{2}M^2
  -\sum_{M<N}\alpha_M\alpha_NMN.
\end{equation}
By \eqref{eq:pj-cj}, binomial expansion modulo $4$ gives
\begin{equation}\label{eq:M-expansion}
  \frac{M(p)-M^2}{2}\equiv
  \sum_{a_r>0}a_r\delta_r\bigl(M/c_r\bigr)^2\pmod2.
\end{equation}
The middle sum in \eqref{eq:Hi-decomposition} consists entirely of square
monomials and therefore cannot contribute to either nonsquare target considered
below. Thus, using $\epsilon(H_{i+1})=\mathcal D(H_i)$, the coefficient
$[T]\epsilon(H_{i+1})$ of a nonsquare monomial $T$ can receive
contributions only from terms of the form
$\alpha_Ma_r\delta_r(M/c_r)^2$ or $\alpha_M\alpha_NMN$ with $M<N$.

Suppose first that $i$ is odd, and put $a=(2^i-2)/6$ and
$T=c_4c_6^{2a}$. If a term $\delta_r(M/c_r)^2$ contributes to
$[T]\epsilon(H_{i+1})$, then some monomial $u$ occurring in $\delta_r$
satisfies $u(M/c_r)^2=T$. Thus $u$ divides $T$ and $T/u$ is a monomial
square. Since every monomial of $\delta_r$ has the form
$c_{r-s}c_{r+s}$, with $c_0=1$, the only monomials of this form that divide $T$,
up to sign, are $c_4$, $c_6$, and $c_4c_6$, occurring in $\delta_2$,
$\delta_3$, and $\delta_5$, respectively. Only $u=c_4$ leaves a square
quotient, since $T/c_4=c_6^{2a}=(c_6^a)^2$, whereas $T/c_6$ and
$T/(c_4c_6)$ are not squares. Hence $r=2$ and $M/c_2=c_6^a$, so
$M=c_2c_6^a$.

No term $MN$ with $M<N$ contributes to $[T]\epsilon(H_{i+1})$. Indeed,
if $MN=T$, then, since $c_4$ occurs in $T$ with exponent one, one of
$M,N$ has the form $c_4c_6^b$. Its degree is
$4+6b\equiv4\pmod6$, contradicting the fact that every monomial of
$H_i$ has degree $2^i\equiv2\pmod6$. Consequently, only $M=c_2c_6^a$ and $r=2$ can contribute to
$[T]\epsilon(H_{i+1})$. By \eqref{eq:Hi-decomposition} and
\eqref{eq:M-expansion}, the corresponding summand before applying
$\epsilon$ is $[c_2c_6^a]H_i\cdot\delta_2c_6^{2a}$. Since
$\epsilon(H_{i+1})=\mathcal D(H_i)$, the uniqueness established above
and $\epsilon(\delta_2)=c_4$ give
\[
  [T]\epsilon(H_{i+1})
  =\overline{[c_2c_6^a]H_i}
  =[c_2c_6^a]\epsilon(H_i).
\]

Now suppose that $i$ is even, and put $a=(2^i-4)/6$ and
$T=c_2c_6^{2a+1}$. As above, if a term $\delta_r(M/c_r)^2$ contributes
to $[T]\epsilon(H_{i+1})$, then some monomial $u$ occurring in
$\delta_r$ satisfies $u\mid T$ and $T/u$ is a square. Up to sign, the only monomials of this form that divide $T$ are
$c_2$, $c_6$, and $c_2c_6$, occurring in
$\delta_1$, $\delta_3$, and $\delta_4$, respectively. Only $u=c_2c_6$
leaves a square quotient, since
$T/(c_2c_6)=c_6^{2a}=(c_6^a)^2$. Hence $r=4$ and
$M/c_4=c_6^a$, so $M=c_4c_6^a$.

Again, no term $MN$ with $M<N$ contributes to
$[T]\epsilon(H_{i+1})$. If $MN=T$, then one of $M,N$ has the form
$c_2c_6^b$, of degree $2+6b\equiv2\pmod6$, whereas every monomial of
$H_i$ has degree $2^i\equiv4\pmod6$. Consequently, only $M=c_4c_6^a$ and $r=4$ can contribute to
$[T]\epsilon(H_{i+1})$. By \eqref{eq:Hi-decomposition} and
\eqref{eq:M-expansion}, the corresponding summand before applying
$\epsilon$ is $[c_4c_6^a]H_i\cdot\delta_4c_6^{2a}$. Since
$\epsilon(H_{i+1})=\mathcal D(H_i)$, the uniqueness established above
and $[c_2c_6]\epsilon(\delta_4)=1$ give
\[
  [T]\epsilon(H_{i+1})
  =\overline{[c_4c_6^a]H_i}
  =[c_4c_6^a]\epsilon(H_i).
\]

Together with $[c_2c_6]\epsilon(H_3)=1$, these two recurrences imply
\[
  [c_2c_6^{(2^i-2)/6}]\epsilon(H_i)=1\quad(i\text{ odd}),
  \qquad
  [c_4c_6^{(2^i-4)/6}]\epsilon(H_i)=1\quad(i\text{ even}).
\]
Finally, \eqref{eq:St1-cj} and \eqref{eq:cartan} show that the only source monomials contributing to the two coefficients in the statement are,
respectively, $c_2c_6^{(2^i-2)/6}$ and $c_4c_6^{(2^i-4)/6}$. In particular,
$\St^1(c_5)=0$, so no $c_6$-factor can be created from a $c_5$-factor. The
assertion follows.
\end{proof}

\section{The \texorpdfstring{$q_i$}{q\_i}-classification and its consequences}
For $i\ge1$, Benson and Wood \cite[Proposition~3.3]{BensonWood}
construct homogeneous polynomials
$F_i\in\mathbb Z[c_2,\ldots,c_{2^i}]$ of degree $2^i$, with $F_1=c_2$,
together with integral polynomials $g_i$ and $h_i$ satisfying
\begin{equation}\label{eq:BW-relation}
  F_i(p)=F_i^2+4z g_i(z,c)+2F_{i+1},
  \qquad
  q_i=2zh_i-F_i.
\end{equation}
The first examples are $q_1=2z^2-c_2$ and
$q_2=2z^4-2z^2c_2+2zc_3-c_4$. Reducing $q_i=2zh_i-F_i$ modulo two gives
$\rho(q_i)=\epsilon(F_i)$. Moreover, dividing the first identity in
\eqref{eq:BW-relation} by two gives
\begin{equation}\label{eq:q-D-recurrence}
  \bigl(F_i(p)-F_i^2\bigr)/2=2zg_i+F_{i+1},
  \qquad \epsilon(F_{i+1})=\mathcal D(F_i).
\end{equation}
Here the second equality is obtained by first reducing coefficients modulo
$2$, which kills $2zg_i$, and then imposing $c_1=0$.

Since $F_2=c_4$,
\eqref{eq:q-D-recurrence} and \eqref{eq:pj-cj} give
$\rho(q_3)=\epsilon(\delta_4)$. Thus, in every rank $m\ge5$,
\begin{equation}\label{eq:Q3}
  \rho(q_3)=c_8+c_2c_6+c_3c_5,
  \qquad
  \St^1(\rho(q_3))=c_9+c_3c_6+c_2c_7.
\end{equation}
Indeed, $\epsilon(\delta_4)=c_8+c_2c_6+c_3c_5$, and the second identity
follows from \eqref{eq:St1-cj} and \eqref{eq:cartan}.

The following lemma provides the two $\Spin(10)$ inputs used below: a Steenrod
obstruction for classes reducing to $c_4$ and the exceptional inclusion
$q_3\in\im\Phi_{10}$.

\begin{lemma}\label{lem:spin10-input}
Let $G=\Spin(10)$. Then
\[
  \St^3(c_4)=c_3c_4+c_2c_5\notin\im\rho,
\]
and $q_3\in\im\Phi_{10}$. Consequently, if
$x\in\CH(BT)^W$ satisfies $\rho(x)=c_4$, then
$x\notin\im\Phi_{10}$.
\end{lemma}

\begin{proof}
In rank $5$, \eqref{eq:steenrod-cj} gives $\St^3(c_4)=c_3c_4+c_2c_5$. This is an odd-degree element of $R_5$ not divisible by $c_5$, so
Lemma~\ref{lem:odd-image-constraint} excludes it from $\im\rho$. For the final assertion, let $x\in\CH(BT)^W$ satisfy
$\rho(x)=c_4$. If $x\in\im\Phi_{10}$, then
\eqref{eq:steenrod-stability-criterion} and the grading of
$\im\rho$ would imply $\St^3(\rho(x))=\St^3(c_4)\in\im\rho$, a contradiction.

For the inclusion $q_3\in\im\Phi_{10}$, Karpenko proved that the Euler
class $c_5$ belongs to $\im\Phi_{10}$
\cite[Lemma~A.1]{KarpenkoModuloTorsion}. Naturality and
\eqref{eq:steenrod-cj} therefore show that $\St^3(c_5)=c_3c_5$ lies in
the image of $\Ch(B\!\Spin(10))\to\Ch(BT)$. Choose a homogeneous class
$\bar a\in\Ch^8(B\!\Spin(10))$ whose restriction is $c_3c_5$, and lift $\bar a$ to a class $a\in\CH^8(B\!\Spin(10))$.

Since, in rank $5$, $\rho(q_3)=c_3c_5$ by \eqref{eq:Q3}, the difference
$d:=q_3-\Phi_{10}(a)$ satisfies $\rho(d)=0$. Thus $d=2y$ for some
$y\in\CH(BT)$. As $d$ is $W$-invariant and $\CH(BT)$ is torsion-free,
$2(wy-y)=0$ for every $w\in W$, and hence $y\in\CH(BT)^W$.

The torsion index of $\Spin(10)$ is $2$ and annihilates the cokernel of
$\Phi_{10}$ \cite[Theorems~0.1 and 1.3(1)]{TotaroTorsion}. Therefore
$2\CH(BT)^W\subset\im\Phi_{10}$, so $d\in\im\Phi_{10}$ and
$q_3\in\im\Phi_{10}$.
\end{proof}

\begin{proof}[Proof of Theorem~\ref{thm:classification}]
For every pair $(n,i)$ in the range of the theorem other than
$(10,3)$, we establish one of the following Steenrod obstructions:
\begin{equation}\label{eq:steenrod-nonliftability}
\begin{aligned}
  \St^3(\rho(q_2))&\notin\im\rho
    &&\text{for }(n,i)\in\{(9,2),(10,2)\},\\
  \St^7(\rho(q_3))&\notin\im\rho
    &&\text{for }(n,i)=(11,3),\\
  \St^1(\rho(q_i))&\notin\im\rho
    &&\text{in all other nonexceptional cases.}
\end{aligned}
\end{equation}

For $i=1$, one has $\rho(q_1)=c_2$, so
$\St^1(\rho(q_1))=c_3$. If $n$ is odd or $4\mid n$,
Lemma~\ref{lem:odd-image-constraint} excludes $c_3$ immediately. If
$n=2m$ with $m$ odd, then $m\ge5$ and $c_m\nmid c_3$, so the same
lemma again gives $c_3\notin\im\rho$.

For $i=2$, one has $\rho(q_2)=c_4$. If $m=4$, then
$2\le r(n)$ forces $n=9$, and \eqref{eq:steenrod-cj} gives
$\St^3(c_4)=c_3c_4\notin\im\rho$ by
Lemma~\ref{lem:odd-image-constraint}. Suppose that $m\ge5$ and $n\ne10$. Then
$\St^1(c_4)=c_5$. If $n$ is odd or $4\mid n$, the same lemma excludes
$c_5$ from $\im\rho$. If $n=2m$ with $m$ odd, then $n\ne10$ implies
$m\ge7$, so $c_m\nmid c_5$, and the lemma again gives the exclusion.
The remaining case $n=10$ is covered by
Lemma~\ref{lem:spin10-input}.

For $i=3$ and $m\ge6$, equation~\eqref{eq:Q3} shows that
$\St^1(\rho(q_3))$ contains $c_3c_6$ with coefficient $1$.
It is therefore a nonzero homogeneous element of odd degree in $R_m$.
If $n$ is odd or $4\mid n$, Lemma~\ref{lem:odd-image-constraint}
excludes it from $\im\rho$. If $n=2m$ with $m$ odd, then $m\ge7$,
and the term $c_3c_6$ shows that the class is not divisible by $c_m$;
the same lemma again gives the exclusion.

It remains for $i=3$ to consider $m=5$, where
\eqref{eq:Q3} gives $\rho(q_3)=c_3c_5$. For $n=11$,
\eqref{eq:St1-cj} gives $\St^1(c_3c_5)=0$.
Equations~\eqref{eq:steenrod-cj} and \eqref{eq:cartan} instead give
\[
  \St^7(c_3c_5)=c_5^3+c_2c_3c_5^2+c_3^2c_4c_5.
\]
The displayed class is a nonzero odd-degree element of $R_5$, so
Lemma~\ref{lem:odd-image-constraint} gives
$\St^7(\rho(q_3))\notin\im\rho$ for $n=11$.
The case $n=10$ is the exceptional inclusion of
Lemma~\ref{lem:spin10-input}.

Finally, let $i\ge4$. Since $i\le r(n)$, one has $m\ge6$.
Applying Lemma~\ref{lem:common-coefficient} to $H_j=F_j$, using \eqref{eq:q-D-recurrence} and
$[c_2c_6]\epsilon(F_3)=1$ from \eqref{eq:Q3}, gives
\[
  [c_3c_6^{(2^i-2)/6}]\St^1(\rho(q_i))=1\quad(i\text{ odd}),
  \qquad
  [c_5c_6^{(2^i-4)/6}]\St^1(\rho(q_i))=1\quad(i\text{ even}).
\]
Since the detected monomials involve only $c_2,\ldots,c_6$ and
$m\ge6$, their coefficients are unchanged after setting $c_j=0$ for
$j>m$. Thus $\St^1(\rho(q_i))$ is a nonzero homogeneous
element of odd degree. Lemma~\ref{lem:odd-image-constraint} excludes it
from $\im\rho$ if $n$ is odd or $4\mid n$. If $n=2m$ with $m$ odd,
then $i\le r(n)=m-2$ and $i\ge4$ imply $m\ge7$, and the relevant
monomial above is not divisible by $c_m$; the same lemma again gives the
exclusion. This completes the proof of
\eqref{eq:steenrod-nonliftability}.

If $q_i\in\im\Phi_n$, then \eqref{eq:steenrod-stability-criterion} and
the grading of $\im\rho$ force every homogeneous component of
$\St(\rho(q_i))$ to belong to $\im\rho$.
Equation~\eqref{eq:steenrod-nonliftability} therefore proves
$q_i\notin\im\Phi_n$ whenever $(n,i)\ne(10,3)$. The exceptional inclusion
$q_3\in\im\Phi_{10}$ is Lemma~\ref{lem:spin10-input}.
\end{proof}

\begin{lemma}\label{lem:bockstein-obstruction}
Let $H=\mathbb G_m\times G$ and $S=\mathbb G_m\times T$. Let
$p:BS\to BT$ be induced by the projection and $j:BS\to BH$ by the
inclusion $S\subset H$. If $1\le i\le r(n)$ and $\ell\ge1$ satisfy
$\St^\ell(\rho(q_i))\notin\im\rho$, then every
$x\in H^{2^{i+1}}(BH;\mathbb Z)$ with $j^*x=p^*q_i$ satisfies
\[
  \beta \Sq^{2\ell}(\bar x)\ne0.
\]
\end{lemma}
\begin{proof}
Let $\omega=\Sq^{2\ell}(\bar q_i)$. Under the $W$-equivariant
cycle-class identifications for $BT$, the map $\rho$ corresponds to
reduction modulo two, while \eqref{eq:brosnan-comparison} identifies
$\St^\ell(\rho(q_i))$ with $\omega$. Hence the hypothesis yields
\begin{equation}\label{eq:topological-nonlift}
  \omega\notin
  \im\!\left(
    H^{2^{i+1}+2\ell}(BT;\mathbb Z)^W
    \longrightarrow H^{2^{i+1}+2\ell}(BT;\mathbb F_2)^W
  \right).
\end{equation}

Let $s:BT\to BS$ be induced by $t\mapsto(1,t)$. Since
$p\circ s=\operatorname{id}_{BT}$, we have $s^*p^*=\operatorname{id}$,
and $s$ is $W$-equivariant. Indeed, if $p^*\omega=\bar\gamma$ for some
$\gamma\in H^{2^{i+1}+2\ell}(BS;\mathbb Z)^W$, then
$\omega=s^*p^*\omega=s^*\bar\gamma=\overline{s^*\gamma}$, while
$s^*\gamma$ is $W$-invariant, contradicting
\eqref{eq:topological-nonlift}. Hence the naturality of Steenrod squares gives
\begin{equation}\label{eq:product-torus-nonlift}
  \Sq^{2\ell}(p^*\bar q_i)=p^*\omega\notin
  \im\!\left(
    H^{2^{i+1}+2\ell}(BS;\mathbb Z)^W
    \longrightarrow H^{2^{i+1}+2\ell}(BS;\mathbb F_2)^W
  \right).
\end{equation}

Suppose that $\beta \Sq^{2\ell}(\bar x)=0$. Exactness of the Bockstein
sequence gives $\eta\in H^{2^{i+1}+2\ell}(BH;\mathbb Z)$ such that
$\bar\eta=\Sq^{2\ell}(\bar x)$. By naturality and $j^*x=p^*q_i$,
\[
  \overline{j^*\eta}
  =j^*\bar\eta
  =j^*\Sq^{2\ell}(\bar x)
  =\Sq^{2\ell}(j^*\bar x)
  =\Sq^{2\ell}(p^*\bar q_i)
  =p^*\omega.
\]
Since $S$ is a maximal torus of $H$, the image of
$j^*:H^*(BH;\mathbb Z)\to H^*(BS;\mathbb Z)$ is contained in
$H^*(BS;\mathbb Z)^W$. Hence $j^*\eta$ is an integral $W$-invariant lift
of $p^*\omega$, contradicting \eqref{eq:product-torus-nonlift}.
\end{proof}

\begin{proof}[Proof of Corollary~\ref{cor:hodge-counterexamples}]
Fix $n\ge7$, put $G=\Spin(n)$, and retain the notation $H,S,p,j$ of
Lemma~\ref{lem:bockstein-obstruction}. Let $\pi:BH\to BG$ be induced by
the projection $H\to G$. For $1\le i\le r(n)$, let $u_i$ be as in
Corollary~\ref{cor:topological-lifts} and choose an integral representative
$\widetilde u_i\in H^{2^{i+1}}(BG;\mathbb Z)$. When $(n,i)=(10,3)$,
take $\widetilde u_3=\cl(a)$ for
$a\in\CH^8(BG)$ with $\Phi_{10}(a)=q_3$, as in
Lemma~\ref{lem:spin10-input}. Put $v_i=\pi^*(\widetilde u_i)$. Since
$u_i$ restricts to $q_i$ and $H^*(BT;\mathbb Z)$ is torsion-free,
every representative $\widetilde u_i$ restricts to $q_i$. Hence $j^*v_i=p^*q_i$. For each nonexceptional pair,
\eqref{eq:steenrod-nonliftability} supplies
$\ell_{n,i}\in\{1,3,7\}$ with
$\St^{\ell_{n,i}}(\rho(q_i))\notin\im\rho$.

Set $k_n=2^{r(n)+1}+15$. Since $H$ is reductive with
positive-dimensional center, Ekedahl's projective approximation theorem
\cite[Theorem~1.3]{EkedahlApproximation} gives a smooth projective complex
variety $X_n$, an algebraic $H$-torsor $P_n\to X_n$, and a classifying
morphism $\theta_n:X_n\to BH$ such that
\[
  \theta_n^*:H^k(BH;\mathbb Z)\xrightarrow{\;\simeq\;}H^k(X_n;\mathbb Z)
  \qquad(k\le k_n).
\]
For $1\le i\le r(n)$, set $\alpha_{n,i}=\theta_n^*v_i$. Since $p$ has
a section, $p^*q_i\ne0$; hence the restriction formula above shows that
$v_i$ is non-torsion, because $H^*(BS;\mathbb Z)$ is torsion-free. As
$2^{i+1}\le k_n$, the class $\alpha_{n,i}$ is also non-torsion.

Each $\alpha_{n,i}$ is an integral Hodge class. Indeed, by
\cite[Theorem~1.3(3)]{TotaroTorsion}, the rational cycle-class map for
$BH$ is an isomorphism. Hence there exists
$b_i\in\CH^{2^i}(BH)\otimes\mathbb Q$ with
$\cl(b_i)=v_i\otimes1$. Naturality gives
$\cl(\theta_n^*b_i)=\alpha_{n,i}\otimes1$, showing that
$\alpha_{n,i}$ has Hodge type $(2^i,2^i)$. In the exceptional case,
$\alpha_{10,3}=\cl(\theta_{10}^*\pi^*a)$ and is algebraic.

Now fix $(n,i)\ne(10,3)$. To prove non-algebraicity modulo torsion, let
$\sigma_X\in H^{2^{i+1}}(X_n;\mathbb Z)$ be an arbitrary torsion class.
Since $\theta_n^*$ is an isomorphism in degree $2^{i+1}$, there is a
unique torsion class $\sigma_H\in H^{2^{i+1}}(BH;\mathbb Z)$ such that
$\theta_n^*\sigma_H=\sigma_X$. Hence $j^*\sigma_H=0$, since
$H^*(BS;\mathbb Z)$ is torsion-free. Setting $x=v_i+\sigma_H$,
Lemma~\ref{lem:bockstein-obstruction} therefore gives
$\beta\Sq^{2\ell_{n,i}}(\bar x)\ne0$.

Since $i\le r(n)$ and $\ell_{n,i}\le7$, the target degree
$2^{i+1}+2\ell_{n,i}+1$ is at most $k_n$. As $\alpha_{n,i}+\sigma_X=\theta_n^*x$, naturality and the
injectivity of $\theta_n^*$ in the target degree give
\[
  \beta\Sq^{2\ell_{n,i}}
    \bigl(\overline{\alpha_{n,i}+\sigma_X}\bigr)
  =\theta_n^*\!\left(
      \beta\Sq^{2\ell_{n,i}}(\bar x)
    \right)
  \ne0.
\]
By \eqref{eq:algebraic-bockstein-vanishing}, the class
$\alpha_{n,i}+\sigma_X$ is not algebraic. Since $\sigma_X$ was arbitrary,
$\alpha_{n,i}$ is not algebraic modulo torsion.
\end{proof}

\section{The \texorpdfstring{$f_i$}{f\_i}-classification for special Clifford groups}

The standard characters $z,x_1,\ldots,x_m$ of $\widetilde T$ give an
identification
\[
  \CH(B\widetilde T)\simeq\mathbb Z[z,x_1,\ldots,x_m].
\]
We use the same symbols for their restrictions to $T$, where
$2z=c_1=x_1+\cdots+x_m$.

Following \cite[\S4, formulas~(4.1)--(4.2)]{KarpenkoMerkurjevSpin}, set
$f_0=2z-c_1=2zg_0+A_0$, where $g_0=1$ and $A_0=-c_1$. Inductively, suppose
that $f_i=2zg_i+A_i$, with
$A_i=f_i|_{z=0}\in\mathbb Z[x_1,\ldots,x_m]$. Write
$A_i=a_1+\cdots+a_t$ in its coefficient expansion as a sum of signed
monomials with unit coefficients, counted with multiplicity and with no
canceling pairs, and define
\begin{equation}\label{eq:fi-square-minus-square}
  f_{i+1}:=2z\bigl(zg_i^2+A_ig_i\bigr)+\sum_{r<s}a_ra_s.
\end{equation}
Equivalently, $2f_{i+1}=f_i^2-\sum_{r=1}^t a_r^2$. The recursion yields homogeneous $W$-invariants $f_i$ of degree $2^i$. Apart from $f_0$, the members of this family occurring in the finite
generating sets are $f_1,\ldots,f_{s(n)}$; see
\cite[Proposition~4.3, Proposition~5.1, and
Corollary~5.2]{KarpenkoMerkurjevSpin}. Throughout, $f_i$ denotes the
polynomial determined by this recursion.

The first two terms are
\[
  f_1=2z^2-2zc_1+c_2,\qquad
  f_2=2z^4-4c_1z^3+2(c_1^2+c_2)z^2-2c_1c_2z+c_1c_3-c_4.
\]
For $i\ge0$, let $\varphi_i\in\CH(BT)^W$ be the image of $f_i$ under
the restriction
$\CH(B\widetilde T)^W\to\CH(BT)^W$. Using $2z=c_1$ on $T$ gives
\[
  \varphi_0=0,
  \qquad \varphi_1=c_2-2z^2=-q_1,
  \qquad \varphi_2=2z^4-2z^2c_2+2zc_3-c_4=q_2.
\]

To describe all reductions, define integral symmetric polynomials
$E_i\in\mathbb Z[c_1,c_2,\ldots]$ by
\begin{equation}\label{eq:clifford-recurrence}
  E_0=c_1,
  \qquad E_{i+1}=\big(E_i^2-E_i(p)\big)/2.
\end{equation}
Here $E_i(p)$ denotes the substitution $c_j\mapsto p_j$ introduced in
Section~3. Inductively, each $E_i$ is an integral homogeneous polynomial of degree
$2^i$. Indeed, since $p_j\equiv c_j^2\pmod2$, one has
$E_i(p)\equiv E_i^2\pmod2$, so the numerator in
\eqref{eq:clifford-recurrence} is coefficientwise divisible by $2$.
Moreover, both squaring and the substitution $c_j\mapsto p_j$ double
degree. The first terms are
\begin{equation}\label{eq:first-Ei}
  E_0=c_1,\qquad E_1=c_2,\qquad E_2=c_1c_3-c_4,
  \qquad
  E_3=\bigl((c_1c_3-c_4)^2-(p_1p_3-p_4)\bigr)/2.
\end{equation}
In rank $m$, we specialize by setting $c_j=0$ for $j>m$.

Let $\widetilde\rho:\CH(B\widetilde T)^W\to\Ch(B\widetilde T)$ be
reduction modulo two.

\begin{lemma}\label{lem:clifford-reduction}
For every $i\ge0$, $\widetilde\rho(f_i)=\overline{E_i}$ and $\rho(\varphi_i)=\epsilon(E_i)$. In particular, $\rho(\varphi_1)=c_2$ and
$\rho(\varphi_2)=c_4$.
\end{lemma}

\begin{proof}
The recursion gives $A_{i+1}=\sum_{r<s}a_ra_s$. Since
$A_0=-c_1=-\sum_jx_j$, one has
$A_1=\sum_{r<s}x_rx_s=c_2=E_1$. Since $A_1$ has nonnegative $x$-monomial coefficients, the recurrence
shows by induction that the same is true of every $A_i$ with $i\ge1$.
Hence the $a_r$ in the coefficient expansion above are monomials with
coefficient $1$, counted with multiplicity, and
$\sum_r a_r^2=A_i(x_1^2,\ldots,x_m^2)$. Therefore
\[
  A_{i+1}=
  (A_i(x_1,\ldots,x_m)^2-A_i(x_1^2,\ldots,x_m^2))/2
  \qquad(i\ge1).
\]
Suppose that $A_i$ is obtained from $E_i$ by substituting
$c_j(x_1,\ldots,x_m)$ for $c_j$. Then
$A_i(x_1^2,\ldots,x_m^2)$ is obtained from $E_i(p)$ by the same
substitution, since
$c_j(x_1^2,\ldots,x_m^2)=p_j$. By
\eqref{eq:clifford-recurrence}, it follows that $A_{i+1}$ is obtained
from $E_{i+1}$ by the same substitution. Since $A_1$ is obtained from $E_1=c_2$ by this substitution, induction
proves the identification for every $i\ge1$. Since
$f_i=2zg_i+A_i$, reduction modulo two gives
$\widetilde\rho(f_i)=\overline{E_i}$ for $i\ge1$, while
$f_0=2z-c_1\equiv c_1=E_0\pmod2$. Finally, restriction to $T$ imposes
$c_1=0$ modulo two, so $\rho(\varphi_i)=\epsilon(E_i)$. The final
assertions follow from \eqref{eq:first-Ei}.
\end{proof}

\begin{lemma}\label{lem:clifford-low-indices}
The following statements hold.
\begin{enumerate}
\item[(i)] If $m\ge4$, then
$\rho(\varphi_3)=c_8+c_2c_6+c_3c_5+c_4^2+c_2c_3^2$.
If $3\le s(n)$, then $\varphi_3\notin\im\Phi_n$.

\item[(ii)] For $n=11$, one has $\varphi_4\notin\im\Phi_{11}$.
\end{enumerate}
\end{lemma}

\begin{proof}
For (i), Lemma~\ref{lem:clifford-reduction},
\eqref{eq:first-Ei}, and \eqref{eq:pj-cj}, together with
$\delta_1=-c_2$ and
$\epsilon(\delta_4)=c_8+c_2c_6+c_3c_5$, give
\[
  \rho(\varphi_3)
  =\epsilon(E_3)
  =\epsilon\!\left(c_4^2-c_1c_3c_4-c_1^2\delta_3-\delta_1c_3^2-2\delta_1\delta_3+\delta_4\right)
  =c_8+c_2c_6+c_3c_5+c_4^2+c_2c_3^2.
\]
Equations \eqref{eq:St1-cj} and \eqref{eq:cartan} yield
$\St^1(\rho(\varphi_3))=c_9+c_3c_6+c_2c_7+c_3^3\ne0$.
If $n$ is odd or $4\mid n$, Lemma~\ref{lem:odd-image-constraint}
shows that this nonzero odd-degree class does not lie in $\im\rho$.
If $n=2m$ with $m$ odd, then $3\le s(n)$ implies $m\ge5$, and the
term $c_3^3$ shows that the class is not divisible by $c_m$; the same
lemma again excludes it from $\im\rho$. Hence the grading of $\im\rho$ and
\eqref{eq:steenrod-stability-criterion} prove (i).

For (ii), since $m=\lfloor 11/2\rfloor=5$, substituting
$p_j=c_j^2+2\delta_j$ into
$E_4=(E_3^2-E_3(p))/2$, then setting $c_j=0$ for $j>5$ and reducing
modulo two, gives
\[
  \rho(\varphi_4)
  =\epsilon(E_4)
  =\epsilon\!\left(\frac{E_3^2-E_3(p)}{2}\right)
  =c_2^4c_4^2+c_2c_3^3c_5+c_2c_4c_5^2+c_3^4c_4+c_3^2c_5^2+c_3c_4^2c_5.
\]
Equations \eqref{eq:St1-cj} and \eqref{eq:cartan} yield
$\St^1(\rho(\varphi_4))=c_2c_5^3+c_3c_4c_5^2\ne0$.
Since $n=11$ is odd, Lemma~\ref{lem:odd-image-constraint}, the grading
of $\im\rho$, and \eqref{eq:steenrod-stability-criterion} prove (ii).
\end{proof}

\begin{lemma}\label{lem:clifford-seven-exception}
For $n=7$, the fixed recursive invariant $f_2$ belongs to
$\im\widetilde\Phi_7$.
\end{lemma}

\begin{proof}
Let $m=3$ and let $(V,q)$ be the split seven-dimensional quadratic
space. By
\cite[Proposition~8.3 and the proof of Proposition~8.4]{BookOfInvolutions},
$C_0(V,q)\simeq\operatorname{End}_k(S)$ for an eight-dimensional vector space $S$.
Since
$\widetilde G=\Gamma^+(V,q)\subset\operatorname{GL}_1(C_0(V,q))$
contains $\Spin(V,q)$ \cite[\S23.A]{BookOfInvolutions}, its action on
$S$ defines an algebraic representation $\Delta$ extending the spin
representation; the scalar subgroup $\mathbb G_m$ acts by scalars.
In the coordinates of \cite[\S2]{KarpenkoEnvelopes}, its
$\widetilde T$-weights are
$\lambda_I=z-\sum_{i\in I}x_i$ for $I\subset\{1,2,3\}$; after
restriction to $T$, the relation $2z=c_1$ identifies them with the
usual spin weights. Pairing $I$ with its
complement, write $b_1,\ldots,b_4$ for the four products
$\lambda_I\lambda_{I^c}$. Explicitly, they are $z(z-c_1)$ and
$(z-x_j)(z-c_1+x_j)$ for $j=1,2,3$. Since
$\lambda_I+\lambda_{I^c}=f_0$, one has
\[
  c_t(\Delta|_{\widetilde T})
  =\prod_{I\subset\{1,2,3\}}(1+\lambda_I t)
  =\prod_{r=1}^4(1+f_0t+b_rt^2).
\]
Summing the four products gives $\sum_{r=1}^4b_r=2f_1$, while the sum of
their six pairwise products is $\sum_{r<s}b_rb_s=f_2+f_1^2$. Hence
\begin{equation}\label{eq:gspin7-c4}
  c_4(\Delta)|_{\widetilde T}
  =f_0^4+3f_0^2\sum_{r=1}^4b_r+\sum_{r<s}b_rb_s
  =f_2+f_1^2+6f_0^2f_1+f_0^4.
\end{equation}
By naturality,
$c_4(\Delta)|_{\widetilde T}
=\widetilde\Phi_7(c_4(\Delta))\in\im\widetilde\Phi_7$.

For a Borel subgroup $\widetilde B\subset\Gamma^+(7)$ and
$B=\widetilde B\cap\Spin(7)$, one has
$\Gamma^+(7)/\widetilde B\simeq\Spin(7)/B$; hence the two groups have
the same torsion index, namely $2$. Thus
\cite[Theorems~0.1 and~1.3(1)]{TotaroTorsion} give
$2\CH(B\widetilde T)^W\subset\im\widetilde\Phi_7$.
Since $f_0$ is the first Chern class of the quotient character,
$p_2$ comes from the vector representation, and
$f_1^2=2f_2+p_2$ by \eqref{eq:fi-square-minus-square}, the classes
$f_1^2$, $6f_0^2f_1=3f_0^2(2f_1)$, and $f_0^4$ lie in the image.
Hence \eqref{eq:gspin7-c4} yields
$f_2\in\im\widetilde\Phi_7$.
\end{proof}

\begin{remark}
For $n=7$, write $\xi$ for the class defined in
\cite[(2.1)]{KarpenkoEnvelopes}. In rank $3$,
\[
  \xi
  =z(z-x_1)(z-x_2)(z-x_1-x_2)
  =z^4-z^2c_2+zc_3,
  \qquad \varphi_2=q_2=2\xi.
\]
Thus the class $f_2$ lifted in
Lemma~\ref{lem:clifford-seven-exception} restricts to $2\xi$, not to
$\xi$, consistently with \cite[Theorem~2.2]{KarpenkoEnvelopes}.
\end{remark}

\begin{proof}[Proof of Theorem~\ref{thm:clifford-classification}]
We first show that $\varphi_i\notin\im\Phi_n$ for every
$(n,i)\ne(7,2)$. For $i=1$, use $\varphi_1=-q_1$ and
Theorem~\ref{thm:classification}. For $i=2$, the identity
$\varphi_2=q_2$ and the same theorem apply for $n\ge9$; the only remaining
cases in the finite range are $n=7,8$, of which the first is exceptional and the
second satisfies
$\St^3(\rho(\varphi_2))=c_3c_4\notin\im\rho$ by
Lemma~\ref{lem:odd-image-constraint}. Hence
\eqref{eq:steenrod-stability-criterion} gives
$\varphi_2\notin\im\Phi_8$. The case $i=3$ and the boundary case
$(n,i)=(11,4)$ are covered by
Lemma~\ref{lem:clifford-low-indices}(i) and (ii), respectively. In every
other case with $i\ge4$, the condition $i\le s(n)$ implies $m\ge6$. By \eqref{eq:clifford-recurrence},
$\mathcal D(E_j)=\epsilon(-E_{j+1})=\epsilon(E_{j+1})$. Hence
Lemma~\ref{lem:common-coefficient}, Lemma~\ref{lem:clifford-reduction},
and $[c_2c_6]\rho(\varphi_3)=1$ give
\[
  [c_3c_6^{(2^i-2)/6}]\St^1(\rho(\varphi_i))=1\quad(i\text{ odd}),
  \qquad
  [c_5c_6^{(2^i-4)/6}]\St^1(\rho(\varphi_i))=1\quad(i\text{ even}).
\]
Since the displayed monomials involve only $c_3,c_5,c_6$ and $m\ge6$,
their coefficients are unchanged after setting $c_j=0$ for $j>m$.
Thus $\St^1(\rho(\varphi_i))$ is a nonzero homogeneous class of odd
degree. Lemma~\ref{lem:odd-image-constraint} excludes this class from
$\im\rho$; when $n=2m$ with $m$ odd, one has $m\ge7$, and the displayed
monomial is not divisible by $c_m$. Hence
\eqref{eq:steenrod-stability-criterion} gives
$\varphi_i\notin\im\Phi_n$.

Now the inclusions $G=\Spin(n)\subset\widetilde G=\Gamma^+(n)$ and
$T\subset\widetilde T$ give the commutative restriction square
\[
\begin{CD}
  \CH(B\widetilde G) @>{\widetilde\Phi_n}>> \CH(B\widetilde T)^W\\
  @VVV @VVV\\
  \CH(BG) @>{\Phi_n}>> \CH(BT)^W.
\end{CD}
\]
Thus $f_i\in\im\widetilde\Phi_n$ would imply
$\varphi_i\in\im\Phi_n$, proving all non-inclusions. The exceptional
inclusion is Lemma~\ref{lem:clifford-seven-exception}.
\end{proof}


\begin{thebibliography}{99}

\bibitem{AntieauTate}
B.~Antieau,
\emph{On the integral Tate conjecture for finite fields and representation
theory},
Algebr. Geom. \textbf{3} (2016), no.~2, 138--149.

\bibitem{AtiyahHirzebruch}
M.~F. Atiyah and F.~Hirzebruch,
\emph{Analytic cycles on complex manifolds},
Topology \textbf{1} (1962), 25--45.

\bibitem{BensonWood}
D.~J. Benson and J.~A. Wood,
\emph{Integral invariants and cohomology of $B\!\Spin(n)$},
Topology \textbf{34} (1995), no.~1, 13--28.

\bibitem{BorelSteenrod}
A. Borel,
\emph{La cohomologie mod $2$ de certains espaces homog\`enes},
Comment. Math. Helv. \textbf{27} (1953), 165--197.

\bibitem{Brosnan}
P. Brosnan,
\emph{Steenrod operations in Chow theory},
Trans. Amer. Math. Soc. \textbf{355} (2003), no.~5, 1869--1903.

\bibitem{EkedahlApproximation}
T.~Ekedahl,
\emph{Approximating classifying spaces by smooth projective varieties},
arXiv:0905.1538v1 [math.AG], 2009.

\bibitem{KamekoCycleMap}
M.~Kameko,
\emph{Representation theory and the cycle map of a classifying space},
Algebr. Geom. \textbf{4} (2017), no.~2, 221--228.

\bibitem{KarpenkoClassifying}
N.~A. Karpenko,
\emph{On classifying spaces of spin groups},
Results Math. \textbf{77} (2022), no.~4, Paper No.~144.

\bibitem{KarpenkoEnvelopes}
N.~A. Karpenko,
\emph{Envelopes and classifying spaces},
Math. Nachr. \textbf{296} (2023), no.~10, 4769--4777.

\bibitem{KarpenkoModuloTorsion}
N.~A. Karpenko,
\emph{On characteristic classes modulo torsion for spin groups},
J. Algebra Appl. \textbf{24} (2025), no.~10, 2550236.

\bibitem{KarpenkoSpecialClifford}
N.~A. Karpenko,
\emph{On special Clifford groups and their characteristic classes},
Ricerche Mat. \textbf{74} (2025), 449--470.

\bibitem{KarpenkoMerkurjevSpin}
N.~A. Karpenko and A.~S. Merkurjev,
\emph{Indexes of generic Grassmannians for spin groups},
Proc. Lond. Math. Soc. (3) \textbf{125} (2022), no.~4, 825--840.

\bibitem{BookOfInvolutions}
M.-A. Knus, A. Merkurjev, M. Rost, and J.-P. Tignol,
\emph{The Book of Involutions},
American Mathematical Society Colloquium Publications, vol.~44,
American Mathematical Society, Providence, RI, 1998.

\bibitem{PirutkaYagita}
A.~Pirutka and N.~Yagita,
\emph{Note on the counterexamples for the integral Tate conjecture over
finite fields},
Doc. Math. Extra Vol. (2015), 501--511.

\bibitem{Primozic}
E. Primozic,
\emph{Motivic Steenrod operations in characteristic $p$},
Forum Math. Sigma \textbf{8} (2020), Paper No.~e52, 25 pp.

\bibitem{TotaroCobordism}
B.~Totaro,
\emph{Torsion algebraic cycles and complex cobordism},
J. Amer. Math. Soc. \textbf{10} (1997), no.~2, 467--493.

\bibitem{TotaroChow}
B. Totaro,
\emph{The Chow ring of a classifying space},
in: Algebraic $K$-theory (Seattle, 1997), Proc. Sympos. Pure Math.
\textbf{67}, Amer. Math. Soc., Providence, RI, 1999, 249--281.

\bibitem{TotaroTorsion}
B. Totaro,
\emph{The torsion index of the spin groups},
Duke Math. J. \textbf{129} (2005), no.~2, 249--290.

\bibitem{WuSquares}
W.~T. Wu,
\emph{On squares in Grassmannian manifolds},
Acta Sci. Sinica \textbf{2} (1953), 91--115.

\end{thebibliography}
\end{document}